\documentclass[12pt,a4paper]{amsart}

\usepackage{amssymb}
\usepackage{amsmath}

\newtheorem{theorem}{THEOREM}

\newtheorem{corollary}[theorem]{COROLLARY}
\newtheorem{lemma}[theorem]{LEMMA}

\newcommand{\I}{\mathrm{I}}
\newcommand{\II}{\mathrm{I\!I}}

\newcommand{\id}{\textrm{id}}

\newcommand{\dd}{\mathrm{d}}
\newcommand{\dA}{\mathrm{det}A}
\newcommand{\is}{\stackrel{\bullet}{=}}
\newcommand{\bnabla}{\overline{\nabla}}

\DeclareMathOperator{\grad}{grad} 

\newcommand{\PP}{\mathcal{P}}

\title[On the Inner Curvature of the Second Fundamental Form\ldots]
{On the Inner Curvature\\ of the Second Fundamental Form\\ of a Surface in the
Hyperbolic Space.}

\author{Steven Verpoort}

\thanks{Partially supported by the Research Foundation -- Flanders (project G.0432.07).}

\keywords{second fundamental form, Gaussian curvature, extrinsic sphere.}
\subjclass[2000]{Primary: 53C42; Secondary: 53C24, 53C45.}

\begin{document}

%%%%%%%%%%%%%%%%%%%%%%%%%%%%%%%%%%%%%%%%%%%%%%%%%%%%%%%%%%%%%%%%%%%%%%%%%%%%%%%%%%%%%%%%%%%%%
\begin{abstract}
The object of study of this article is compact surfaces in the three-dimensional hyperbolic space with a positive-definite second fundamental form. 
It is shown that several conditions on the Gaussian curvature of the second fundamental form can be satisfied only by extrinsic spheres.
\end{abstract}
%%%%%%%%%%%%%%%%%%%%%%%%%%%%%%%%%%%%%%%%%%%%%%%%%%%%%%%%%%%%%%%%%%%%%%%%%%%%%%%%%%%%%%%%%%%%%

\maketitle

%%%%%%%%%%%%%%%%%%%%%%%%%%%%%%%%%%%%%%%%%%%%%%%%%%%%%%%%%%%%%%%%%%%%%%%%%%%%%%%%%%%%%%%%%%%%%

\section{Introduction.}
%%%%%%%%%%%%%%%

A substantial number of results in classical differential geometry characterise the spheres among the compact surfaces in the Euclidean three-dimensional space with a positive-definite second fundamental form
by means of various curvature conditions.
The study of the second fundamental form of such surfaces
from the \emph{metrical} point of view has also received considerable attention, and many characterisations of the spheres in terms of the intrinsic curvature of the second fundamental form are known (e.g.\ \cite{baik1987a}--\cite{simon1976}). 

It has been noticed already in \cite{aledoromero2003} 
that several of these results can be generalised to surfaces in the de Sitter space. In the present article, some of these theorems will likewise be adapted to surfaces in the three-dimensional hyperbolic space 
$\mathrm{H}^3$. 

%%%%%%%%%%%%%%%

\section{Summary of the Adopted Notation and Useful Formulae.}
%%%%%%%%%%%%%%%

For a compact surface $M\subseteq\mathrm{H}^3$ endowed with a unit normal vector field $N$, the corresponding shape operator is given by
\[
A  : \mathfrak{X}(M)\mapsto\mathfrak{X}(M)
   : V \mapsto  -\bnabla_{V} N \,.  
\]
Here $\bnabla$ is the Levi-Civita connection of the enveloping space $\mathrm{H}^3$, and $\mathfrak{X}(M)$ stands for the collection of all tangent vector fields on the surface $M$. Half of the trace of this operator is defined as the mean curvature $H$ of the surface, whereas the determinant of this operator is related to the Gaussian curvature (of its first fundamental form $\I$, which is the restriction of the metric of the enveloping space $\mathrm{H}^3$ to the surface) through the contracted Gauss equation $\dA = K+1$.
If the shape operator is a constant multiple of the identity, the surface will be said to be an extrinsic sphere.

The second fundamental form of a compact surface $M\subseteq\mathrm{H}^3$ is defined by
\[
\II :\mathfrak{X}(M)\times\mathfrak{X}(M)\rightarrow\mathfrak{F}(M) : (V,W)\mapsto \II(V,W)=\I\left(A(V),W\right)\,.
\]
Here $\mathfrak{F}(M)$ stands for the set of all real-valued functions defined on $M$.

On a surface in $\mathrm{H}^3$ with strictly positive Gaussian curvature, a
unit normal vector field can always globally be chosen in such a way that the second fundamental form becomes positive-definite. This will implicitly be assumed for such surfaces.

The focus of the present article lies on compact (immersed) surfaces $M$ in $\mathrm{H}^3$ of which the second fundamental form is positive-definite, and can thus be seen as a Riemannian metric. The Gaussian curvature of this two-dimensional Riemannian space $(M,\II)$ will be denoted by $K_{\II}$, and is given by
\cite{aledoaliasromero2005}
\[
K_{\II} = \frac{H\, K}{K+1} +\PP- \frac{\I\left(\grad K\,,\,A^{\leftarrow}\grad K\right)}{8(K+1)^2}\,,
\]
where $\PP$ is a certain non-negative function and $A^{\leftarrow}$ stands for the inverse of the shape operator.

It should also be remarked that Erard's 
formula (Q) \cite{erard1968} can be
adapted for surfaces in $\mathrm{H}^3$ with a positive-definite second fundamental form:
\begin{eqnarray*}
\frac{\dA}{2}\,\I\left(\grad\left(\frac{H^2}{\dA}\right),A^{\leftarrow}\grad H \right)
-\frac{1}{4}\,\I\left(\grad\left(\frac{H^2}{\dA}\right),\grad \dA\right)\hspace{0.0cm}\\
\textrm{(Q)}\hspace{5.8cm}=\,
2 H \left(K_{\II}- \frac{H K}{\dA} \right)(H^2-\dA)\,.
\end{eqnarray*}
It is straightforward that the above equation is satisfied at the umbilical points of the surface. The equation can be established in a region which is free of umbilical points, by expanding both the left- and the right-hand side in terms of the orthonormal basis of principal directions.

The area element of such a surface will be denoted by $\dd\Omega$,
whereas the area element of the second fundamental form is given by $\dd\Omega_{\II}=\sqrt{\dA}\,\dd\Omega$.

%%%%%%%%%%%%%%%
\section{Characterisations of Extrinsic Spheres.}
%%%%%%%%%%%%%%%

In \cite{aledoaliasromero2005} a consideration of the formula for $K_{\II}$ resulted in a new proof of the following classical result:

\begin{theorem}
\label{thm:ch4_Kcst}
A compact surface in $\mathrm{H}^3$ has constant Gaussian curvature if and only if it is an extrinsic sphere.
\end{theorem}

Some characterisations of extrinsic spheres in $\mathrm{H}^3$ by means of the curvature of the second fundamental form can now be given. 

The theorem below should be compared with the main result of \cite{schneider1972}.

\begin{theorem}
\label{thm:ch4_KII_hyp}
Let $M$ be a compact surface in the hyperbolic space $\mathrm{H}^3$, and assume that the second fundamental form of $M$ is positive-definite. The Gaussian curvature of the second fundamental form is constant 
if and only if
$M$ is an extrinsic sphere.
\end{theorem}
\begin{proof}
If a surface in $\mathrm{H}^3$ is given, which satisfies the mentioned conditions and most notably has a second fundamental form of constant curvature $K_{\II}$, choose a point $p_{+}$ where $K$ attains its maximal value. 
It is well-known that this maximal value $\left.K\right\vert_{p_{+}}$ has to be strictly positive (see, e.g., pp.\ 1819 ff.\ of \cite{aledoaliasromero2005}). Thus every $q\in M$ satisfies
\begin{eqnarray*}
\left.K_{\II}\right\vert_{(q)}
=
\left.K_{\II}\right\vert_{(p_{+})}
&\geqslant&
\left\lgroup
\frac{H}{\sqrt{\det A}}\, \frac{K}{\sqrt{K+1}}
\right\rgroup_{(p_+)}
\\
&\geqslant&
\left\lgroup
\frac{K}{\sqrt{K+1}}
\right\rgroup_{(p_+)}
\geqslant
\left\lgroup
\frac{K}{\sqrt{\dA}}
\right\rgroup_{(q)}
\end{eqnarray*}
(where the fact that the function 
$\left]\,-1\,,\,+\infty\,\right[\rightarrow\mathbb{R}: x \mapsto \frac{x}{\sqrt{1+x}}$ is increasing has been employed). It can be concluded that
\[
\int K_{\II}\,\dd\Omega_{\II} =
\int K_{\II}\,\sqrt{\dA}\ \dd\Omega
\geqslant 
\int K \,\dd\Omega =
\int K_{\II}\,\dd\Omega_{\II}
\,.
\]
This is only possible if 
the equality $K_{\II}=\frac{K}{\sqrt{1+K}}$ is satisfied on $M$. In particular, $K$ is constant, and the result follows from the previous theorem.

Conversely, it is plain that the function $K_{\II}$ is constant for an extrinsic sphere.
\end{proof}
The following theorem is reminiscent to \cite{koutroufiotis1974}.
\begin{theorem}
\label{thm:ch4_koutroufiotisH3}
Assume $M\subseteq\mathrm{H}^3$ is a compact surface in the hyperbolic space with strictly positive Gaussian curvature. Then $C K = K_{\II}$ (for some $C\in\mathbb{R}$) if and only if $M$ is an extrinsic sphere.
\end{theorem}
\begin{proof}
First of all, it should be remarked that $C>0$. This is an immediate consequence of the Gauss--Bonnet theorem.

Let now $p_{-}$ be a point of $M$ where the function $\dA = K+1$ attains its minimal value. The given condition implies that
\begin{eqnarray*}
\left.C\,K\right\vert_{(p_{-})}
=\left.\,K_{\II}\right\vert_{(p_{-})}
&=&\left\lgroup
\frac{H\,K}{K + 1} + \PP - 
\frac{\I\left(\grad K\,,\,A^{\leftarrow}\grad K \right)}{8 (K+1)^2}
\right\rgroup_{(p_{-})} \\
&\geqslant&
\left\lgroup
\frac{H\,K}{K + 1} \right\rgroup_{(p_{-})}\,.
\end{eqnarray*}
Due to the fact that $\left.K\right\vert_{(p_{-})}$ is strictly positive,
it follows that 
\[
\left.\rule{0pt}{12pt}C\left( K + 1\right)\,\right\vert_{(p_{-})}
\geqslant
\left.H\right\vert_{(p_{-})}
\geqslant
\left.\rule{0pt}{12pt}\sqrt{ K  + 1}\right\vert_{(p_{-})}\,,
\]
and the conclusion is that every $q\in M$ satisfies
\[
C\left.\sqrt{K + 1}\right\vert_{(q)}
\geqslant
C\left.\sqrt{K + 1}\right\vert_{(p_{-})}
\geqslant 
1\,.
\]
This means that $K_{\II} \sqrt{\dA} =
C K \sqrt{K+1} \geqslant K$, and by integration there results
\[
\int
K_{\II}
\,\dd \Omega_{\II}
=
\int
K_{\II} \sqrt{\dA}
\ \dd \Omega
\geqslant
\int K \,\dd \Omega\,.
\]
Thus the equality $C K_{\II} \sqrt{K+1} =C K = K_{\II}$ is valid, which can only be the case if $K$ is a constant. This finishes the proof.
\end{proof}
The next lemma will enable us to generalise a result of \cite{simon1976} in the subsequent theorem.
\begin{lemma}
Let $M\subseteq\mathrm{H}^3$ be a compact surface with positive-definite second fundamental form. If a point $p$ can be found where $K_{\II}$ has a global minimum while $K$ has a global maximum, then $M$ is an extrinsic sphere.
\end{lemma}
\begin{proof}
It is known that $K_{(p)}$ is strictly positive, and it follows in this way that every point $q\in M$ satisfies
\begin{eqnarray*}
\left.K_{\II}\right\vert_{(q)}
\geqslant
\left.K_{\II}\right\vert_{(p)}
&\geqslant&
\left\lgroup\frac{H}{\sqrt{\dA}}\ \frac{K}{\sqrt{K+1 }}\right\rgroup_{(p)}
\\
&\geqslant&
\left\lgroup\frac{K}{\sqrt{K+1 }}\right\rgroup_{(p)}
\geqslant
\left\lgroup\frac{K}{\sqrt{K+1 }}\right\rgroup_{(q)}
\,.
\end{eqnarray*}
A twofold application of the Gauss--Bonnet theorem thus gives us that the integral of a non-negative function is zero:
\[
0 = \int \left\lgroup K_{\II}-\frac{K}{\sqrt{K+1}} \right\rgroup\,\dd\Omega_{\II}\,.
\]
Consequently, this integrand vanishes identically, and every inequality in the above reasoning is an equality. This means that $K_{\II}$ is constant, such that $M$ is an extrinsic sphere.
\end{proof}
\begin{theorem}
\label{thm:ch4_thmsimon}
Let $M\subseteq\mathrm{H}^3$ be a compact 
surface with a positive-definite second fundamental form. Assume that the 
condition $F(K,K_{\II})=0$ is fulfilled on $M$ for a function
$F:\mathbb{R}^2\rightarrow\mathbb{R}:(u,v)\mapsto F(u,v)$ which satisfies the following requirements:
\[
\left\{
\begin{array}{c}
\displaystyle F_u > 0 \,;\\
\displaystyle F_v \geqslant 0\,;
\end{array}
\right.
\qquad
\textrm{or}
\qquad
\left\{
\begin{array}{c}
\displaystyle F_u \geqslant 0 \,;\\
\displaystyle F_v > 0\,.
\end{array}
\right.
\] 
Then $M$ is an extrinsic sphere.
\end{theorem}

(In particular, this condition is satisfied if $K=f(K_{\II})$ or $K_{\II}=f(K)$ for a decreasing function $f$.)

\begin{proof}
Suppose that a surface $M\subseteq\mathrm{H}^3$ satisfies the conditions as formulated in the theorem for some function $F$, including the first set of requirements on $F_u$ and $F_v$. 
Let $q\in M$ be a point where $K$ achieves its global maximum, and $r\in M$ a point where $K_{\II}$ achieves a global minimum. 
Assume that $M$ is not an extrinsic sphere. According to the previous lemma, there necessarily holds  $K_{(q)}>K_{(r)}$, and hence
\[
0
= 
F\left(K_{(q)},\left.K_{\II}\right.\vert_{(q)}\right)
>
F\left(K_{(r)},\left.K_{\II}\right.\vert_{(q)}\right)
\geqslant
F\left(K_{(r)},\left.K_{\II}\right.\vert_{(r)}\right)
=0\,.
\] 
This is clearly a contradiction.
The case in which the second set of conditions is satisfied is similar.
\end{proof}
The following result is similar to \cite{koutroufiotis1974}.
\begin{theorem}
\label{thm:ch4_koutroufiotis}
A compact surface $M\subseteq\mathrm{H}^3$ with strictly positive Gaussian curvature is either an extrinsic sphere, or the function $\displaystyle\rule{0pt}{18pt} K_{\II}-\frac{K}{\sqrt{\dA}}$ 
changes sign.
\end{theorem}
\begin{proof}
Assume for a compact surface $M\subseteq\mathrm{H}^3$ with strictly positive Gaussian curvature, which is not an extrinsic sphere, the function $K_{\II}-\frac{K}{\sqrt{\dA}}$ does not change sign. The latter condition can only be satisfied if this function vanishes identically. Let $p_{\bullet}$ be a point where the function $\frac{H^2}{\dA}$ assumes its maximal value. Formula \textrm{(Q)} reduces at this point $p_{\bullet}$ to
\[
0 \is  2 H \left( H^2-\dA\right)\left(K_{\II}-\frac{H\,K}{\dA} \right) \,.
\]
(The symbol $\is$ indicates that both sides of an equation should be evaluated at the point $p_{\bullet}$.) 
Since $M$ is not an extrinsic sphere, $p_{\bullet}$ is non-umbilical, and thus it can be inferred that
\[
K_{\II} \is \frac{H\, K}{\dA} \stackrel{\bullet}{>} \frac{K}{\sqrt{\dA}}\,,
\] 
which is clearly a contradiction.
\end{proof}

\begin{corollary}
\label{cor:ch4_sgn}
A compact surface $M\subseteq\mathrm{H}^3$ with strictly positive Gaussian curvature is a sphere as soon as any of the following conditions is satisfied:
\[
\left\{
\begin{array}{lcrcl}
\textrm{(i.)}   &\hspace{3mm}& \displaystyle K_{\II} &\displaystyle \geqslant&\displaystyle \frac{K}{\sqrt{K+1}}\,;\raisebox{-16pt}{\rule{0pt}{35pt}} \\
\textrm{(ii.)}  &\hspace{3mm}& \displaystyle K_{\II} &\displaystyle \leqslant&\displaystyle \frac{K}{\sqrt{K+1}}\,;\raisebox{-16pt}{\rule{0pt}{35pt}}\\
\textrm{(iii.)} &\hspace{3mm}& \displaystyle K_{\II} &\displaystyle \leqslant&\displaystyle \frac{K}{H}\,.\raisebox{-11pt}{\rule{0pt}{30pt}}
\end{array}
\right.
\]
\end{corollary}

The following theorem generalises theorem 29 of \cite{aledoromero2003}.
\begin{theorem}
\label{thm:ch4_HKKII}
Let $M\subseteq\mathrm{H}^3$ be a compact surface with strictly positive Gaussian curvature. Assume real numbers $C$, $r$ and $s$ can be found,
subject to the conditions $0\leqslant s \leqslant 1$, $r\leqslant 1$, and $2 r + s \leqslant 1$, such that the equation 
\[
K_{\II} = C\,H^s\,K^r
\]
is satisfied. Then $M$ is an extrinsic sphere. 
\end{theorem}
\begin{proof}
Let $\varphi$ stand for the function
\[
\varphi : \left]\,0\,,\,+\infty\,\right[\rightarrow\mathbb{R} : 
x \mapsto \frac{x^{1-r} }{(1+x)^{\frac{s+1}{2}}}\,.
\]
It should be remarked that $\varphi'(x)\geqslant 0$ for all $x\in \left]\,0\,,\,+\infty\,\right[$. Furthermore, the constant $C$ has to be strictly positive.
Let $K$ assume its maximum in a point $p_+$. Then it can be concluded
\begin{eqnarray*}
\left(C\, H^s\, K^r\right)_{(p_+)} 
&=& \left.K_{\II}\right\vert_{p_+} \\
&=& \left\lgroup \frac{H\ K}{K+1}+\PP-\frac{\I\left(\grad K\,,\,A^{\leftarrow}\grad K\right)}{8(K+1)^2} \right\rgroup_{(p_+)}\\
&\geqslant& 
\left\lgroup \frac{H K}{K+1} \right\rgroup_{(p_+)}\,.
\end{eqnarray*}
Since $s-1\leqslant 0$, it follows that $\displaystyle \left(K+1\right)^{\frac{s-1}{2}} \geqslant H^{s-1}$. Consequently,
\[
\left\lgroup C\, (K+1)^{\frac{s-1}{2}}\, K^{r-1}\right\rgroup_{(p_+)}
\geqslant 
\left\lgroup C\, H^{s-1}\, K^{r-1} \right\rgroup_{(p_+)}
\geqslant
\left\lgroup\frac{1}{K+1}\right\rgroup_{(p_+)}
\]
and hence for all $q\in M$
\[
C \geqslant 
\left\lgroup \frac{K^{1-r} }{(K+1)^{\frac{s+1}{2}}} \right\rgroup_{(p_+)}
=
\left.\varphi(K)\right\vert_{(p_+)}\geqslant\left.\varphi(K)\right\vert_{(q)}
=
\left\lgroup \frac{K^{1-r} }{(K+1)^{\frac{s+1}{2}}} \right\rgroup_{(q)}
\,.
\]
Thus there holds
\[
K_{\II}\sqrt{K+1}
= C\,H^s\,K^r\,\sqrt{K+1}
\geqslant
\frac{H^s\, K}{(K+1)^{\frac{s}{2}} }\geqslant K\,.
\]
The theorem follows by virtue of corollary~\ref{cor:ch4_sgn}.
\end{proof}

The lemma below is a preparation for theorem~\ref{thm:ch4_HKKIIK+1}, which generalises a theorem of \cite{hasanis1982}. 
\begin{lemma}
\label{lem:ch4_Has}
Assume a compact surface $M\subseteq\mathrm{H}^3$ has positive-definite second fundamental form. If $p_{\bullet}\in M$ is a critical point of $H$, then 
$\displaystyle
\left.K_{\II}\right\vert_{(p_{\bullet})} \geqslant \left\lgroup\frac{H\,K}{\dA}\right\rgroup_{(p_{\bullet})}
$.
\end{lemma}
\begin{proof}
Let the point $p_{\bullet}$ be such as in the lemma and suppose that the inequality
$\displaystyle
K_{\II} \stackrel{\bullet}{<} \frac{H\,K}{\dA}$ holds.
Since the gradient of $ H $ vanishes at the point $p_{\bullet}$, formula (Q) implies the following at the point $p_{\bullet}$:
\[
0 \stackrel{\bullet}{\leqslant}
\frac{H^2}{4(\det A)^2} \ 
\I\left(\grad\dA\,,\,\grad\dA\right)
\is
2\,H\,(H^2-\dA)\,\left(K_{\II}-\frac{H\,K}{\dA}\right)\,.
\]
Furthermore,  $2\,H\,(H^2-\dA)$ is non-negative. The above inequality can only be valid if $H^2\is \dA$, and consequently $\left.\grad\dA\right\vert_{(p_{\bullet})}$ vanishes. Thus, the following contradiction follows:
\[
K_{\II}
\is
\frac{H\,K}{\dA } +\PP
\stackrel{\bullet}{\geqslant}
\frac{H\,K}{\dA }
\,.
\]
\end{proof}

\begin{theorem}
\label{thm:ch4_HKKIIK+1}
Let $M\subseteq\mathrm{H}^3$ be a compact surface with strictly positive Gaussian curvature. Assume real numbers $C$, $r$ and $s$ can be found,
subject to the condition $-1\leqslant r \leqslant \frac{-1}{2}$, such that the equation
\[
K_{\II} = C\,H^s\,(K+1)^r\,K
\]
is satisfied. Then $M$ is an extrinsic sphere.
\end{theorem}
\begin{proof}
It will first be shown that 
\begin{equation}
\label{eq:ch4_H}
1\leqslant C \,H^{s+2r+1}\,.
\end{equation}

\begin{itemize}
\item[1.] \textit{First Case: $s+2r+1\geqslant 0$.}
Let $p_-$ be a point where $H$ achieves its minimum, and choose an arbitrary point $q\in M$. 
It is known that
\[
H_{(p_-)}\leqslant \left\lgroup\frac{\dA }{K}K_{\II}  \right\rgroup_{(p_-)}
=
\left\lgroup
C\,H^s\,(K+1)^{1+r}
\right\rgroup_{(p_-)}
\]
and hence
\begin{eqnarray*}
1&\leqslant& C \left\lgroup H^{s-1} (K+1)^{1+r} \right\rgroup_{(p_-)}\\
&\leqslant&C \left\lgroup H^{s-1+2+2r}\right\rgroup_{(p_-)} \leqslant C \left.H^{s+2r+1}\right\vert_{(q)}\,.
\end{eqnarray*}
\item[2.] \textit{Second Case: $s+2r+1\leqslant 0$.} 
This follows similarly by investigating a point where $H$ assumes its maximum.
\end{itemize}
It follows from (\ref{eq:ch4_H}) that
\[
(K+1)^{\frac{-1-2r}{2}} \leqslant H^{-1-2r} =\frac{1}{H^{1+2r}} \leqslant C H^s\,,
\]
and hence also
\[
K \leqslant C\, H^s \, (K+1)^{r}\, K \, \sqrt{K+1} = K_{\II}\, \sqrt{K+1}\,.
\]
The result now follows from corollary \ref{cor:ch4_sgn}.
\end{proof}

\begin{lemma}
\label{lem:ch4_KII_HdetA}
Let $M\subseteq\mathrm{H}^3$ be a surface with positive-definite second fundamental form. If $p_{\bullet}\in M$ is a critical point of $\frac{H}{\dA}$, there holds 
\[
\left.K_{\II}\right\vert_{(p_{\bullet})}\geqslant 
\left\lgroup \frac{H\, K}{\dA}\right\rgroup_{(p_{\bullet})}.
\]
\end{lemma}

\begin{proof}
Since $p_{\bullet}$ is a critical point of $\frac{H}{\dA}$, the formula
\[
\grad \left(\frac{H^2}{\dA}\right) 
\is
\frac{H}{\dA} \grad H
\]
and similar ones are valid.
In this way, formula (Q) can be rewritten as follows at the point $p_{\bullet}$:
\[
\frac{1}{4}\, \I \left(\rule{0pt}{10pt}
\grad H\,,
\left\{ 2 H \id- A  \right\}\circ \left\{A^{\leftarrow}\right\}
\grad H
\right)
\is
2 H \left( H^2-\dA\right)\left( K_{\II}-\frac{H\,K}{\dA} \right)\,.
\]
Since both operators between curly brackets in the above formula are 
positive-definite, it follows that
\[
0 \stackrel{\bullet}{\leqslant}
2 H \left( H^2-\dA\right)\left( K_{\II}-\frac{H\,K}{\dA} \right)
\,.
\]
Assume first that the inequality $K_{\II} \stackrel{\bullet}{\geqslant} 
\frac{H\,K}{\dA}$ is not satisfied. In regard of the above inequality, this can only be the case if $H^2\is \dA$. But now the rewritten formula (Q) reveals that $\grad H\is 0$ and hence also $\grad \dA\is0$. 
Consequently,
\[
K_{\II}
\is
\frac{H\,K}{\det A}
+\PP
\stackrel{\bullet}{\geqslant}
\frac{H\,K}{\dA}
\,,
\]
which is in contradiction with our assumption. This finishes the proof.
\end{proof}

The next two theorems generalise results of \cite{baik1987a}.
\begin{theorem}
\label{thm:ch4_baik}
Let $M\subseteq\mathrm{H}^3$ be a compact surface with strictly positive Gaussian curvature. If the equality
\[
K_{\II} = C\, H^s\, (K+1)^r\, K
\]
is satisfied for real numbers $C$, $s$ and $r$, subject to the condition $\frac{-1}{2}\leqslant r+s \leqslant 0$, then is $M$ an extrinsic sphere.
\end{theorem}

\begin{proof}
It is first remarked that, for a critical point $p_{\bullet}$ of $\frac{H}{\dA}$,
there holds $1\stackrel{\bullet}{\leqslant} C \left(\frac{H}{\dA}\right)^{-1-s-2r}$. Namely, an application of the previous lemma gives that
\[
C\, H^s\,(K+1)^r\,K \is K_{\II} \stackrel{\bullet}{\geqslant} \frac{H\, K}{K+1}
\]
and hence
\begin{eqnarray*}
1 &\stackrel{\bullet}{\leqslant}& C\, H^{s-1}\,(K+1)^{r+1}\\
&\is& C\,H^{s-1}\,(K+1)^{r+1+(s+r)}\,(K+1)^{-(s+r)}\\
&\stackrel{\bullet}{\leqslant}& C\, H^{s-1-2(s+r)}\,(K+1)^{r+1+(s+r)}\\
&\is&C\, H^{-1-s-2r} \, (K+1)^{1+s+2r}\\
&\is& C \left(\frac{H}{K+1}\right)^{-1-s-2r}\,.
\end{eqnarray*}
It is now an easy consequence that the equality
\begin{equation}
\label{eq:ch4_HKplus1}
1 \leqslant C \left(\frac{H}{\dA}\right)^{-1-s-2r}
\end{equation}
holds on the entire surface $M$.

\begin{itemize}
\item[1.] \textit{First Case: $1+s+2r\leqslant 0$.}
Let now $p_-$ be a point where $\frac{H}{\dA}$ assumes its minimum.  There holds, for every point $q\in M$,
\[
1 \leqslant C\left\lgroup \frac{H}{K+1}\right\rgroup^{-1-s-2r}_{(p_-)} \leqslant 
C \left\lgroup\frac{H}{K+1}\right\rgroup^{-1-s-2r}_{(q)}\,.
\]
\item[2.] \textit{Second Case: $1+s+2r\geqslant 0$.}
This follows similarly by investigating a point where $\frac{H}{\dA}$ assumes its maximum.
\end{itemize}
It now follows from (\ref{eq:ch4_HKplus1}) that
\begin{eqnarray*}
K &\leqslant&
C\,H^{-1-s-2r}\,(K+1)^{1+s+2r}\,K\\
&=&
C\,H^s\,(K+1)^r\,K\,H^{-1-2s-2r}\,(K+1)^{\frac{1}{2}+s+r}\,\sqrt{K+1}\\
&=&
K_{\II}\left(\frac{H^2}{\dA} \right)^{-\frac{1}{2}-s-r}\,\sqrt{K+1}\\
&\leqslant&
K_{\II}\,\sqrt{K+1}\,.
\end{eqnarray*}
The theorem follows from corollary~\ref{cor:ch4_sgn}.
\end{proof}

\begin{lemma}
\label{lem:ch4_baik}
Let $p$ be a point of a compact surface $M\subseteq\mathrm{H}^3$ with $K>0$ where
$\frac{K_{\II}}{K}$ achieves its absolute minimum and $\frac{H}{\dA}$ achieves its absolute maximum. Then $M$ is an extrinsic sphere.
\end{lemma}
\begin{proof}
Under the imposed assumptions, every point $q$ of $M$ satisfies
\[
\left.\frac{K_{\II}}{K}\right\vert_{(q)}
\geqslant
\left.\frac{K_{\II}}{K}\right\vert_{(p)}
\geqslant
\left.\frac{H}{\dA} \right\vert_{(p)}
\geqslant
\left.\frac{H}{\dA} \right\vert_{(q)}\,.
\]
The result follows at once from  corollary~\ref{cor:ch4_sgn}.
\end{proof}

\begin{theorem}
\label{thm:ch4_baik_F}
Let $M\subseteq\mathrm{H}^3$ be a compact surface with strictly positive Gaussian curvature. Assume that the condition $F(\frac{H}{\dA},\frac{K_{\II}}{K})=0$ is fulfilled for a function $F:\mathbb{R}^2\rightarrow\mathbb{R}:(u,v)\mapsto F(u,v)$ which satisfies the following requirements:
\[
\left\{
\begin{array}{c}
\displaystyle F_u > 0 \,;\\
\displaystyle F_v \geqslant 0\,;
\end{array}
\right.
\qquad
\textrm{or}
\qquad
\left\{
\begin{array}{c}
\displaystyle F_u \geqslant 0 \,;\\
\displaystyle F_v > 0\,.
\end{array}
\right.
\] 
Then $M$ is an extrinsic sphere. 
\end{theorem}

\begin{proof}
Similar to the proof of theorem \ref{thm:ch4_thmsimon}.
\end{proof}

\begin{flushright}
\begin{it}
\begin{footnotesize}
Address:\\ 
K.U.Leuven\\ 
Departement Wiskunde\\ 
Afdeling Meetkunde\\
Celestijnenlaan 200B bus 2400\\
3001 Heverlee\\ 
Belgium\\
$ $ \\
E-mail:\\
\verb|steven.verpoort@wis.kuleuven.be|\\
\end{footnotesize}
\end{it}
\end{flushright}

\end{document}